\documentclass[reqno,a4paper, 11pt]{amsart}

\usepackage[a4paper=true,pdfpagelabels]{hyperref}
\usepackage{graphicx}
\usepackage{latexsym}
\usepackage{mathabx}
\usepackage{amsmath}
\usepackage{amssymb}

\usepackage[T1]{fontenc}
\usepackage[utf8]{inputenc}
\usepackage{lmodern}

\usepackage{color}   

\newtheorem{theorem}{Theorem}
\newtheorem{lemma}[theorem]{Lemma}
\newtheorem{corollary}[theorem]{Corollary}
\newtheorem{proposition}[theorem]{Proposition}

\newtheorem{lettertheorem}{Theorem}
\newtheorem{letterlemma}[lettertheorem]{Lemma}

\theoremstyle{definition}

\theoremstyle{remark}

\numberwithin{equation}{section}

\newcommand{\set}[1]{\left\{#1\right\}}
\newcommand{\abs}[1]{\lvert#1\rvert}
\newcommand{\nm}[1]{\lVert#1\rVert}

\newcommand{\B}{\mathcal{B}}
\newcommand{\D}{\mathbb{D}}
\newcommand{\DD}{\widehat{\mathcal{D}}}
\newcommand{\Dd}{\widecheck{\mathcal{D}}}

\newcommand{\DDD}{\mathcal{D}}

\newcommand{\N}{\mathbb{N}}

\renewcommand{\phi}{\varphi}

\def\a{\alpha}       \def\b{\beta}        \def\g{\gamma}
           
     \def\om{\omega}      
              
                  \def\z{\zeta}

\def\omg{{\widehat{\omega}}}
\def\nug{{\widehat{\nu}}}
\def\mug{{\widehat{\mu}}}

\renewcommand{\H}{\mathcal{H}}

\newenvironment{Prf}{\noindent{\emph{Proof of}}}
{\hfill$\Box$ }

\begin{document}

\title[ Bergman projection induced by radial weight acting on growth spaces]{ Bergman projection induced by radial  weight acting on growth spaces}
	

\author{\'Alvaro Miguel Moreno}
\address{Departamento de Analisis Matem\'atico, Universidad de M\'alaga, Campus de Teatinos, 
29071 Malaga, Spain}
\email{alvarommorenolopez@uma.es}
\author{Jos\'e \'Angel Pel\'aez}
\address{Departamento de Analisis Matem\'atico, Universidad de M\'alaga, Campus de Teatinos, 
29071 Malaga, Spain}
\email{japelaez@uma.es}
\author{Jari Taskinen}
\address{University of Helsinki, Deparment of Mathematics and Statistics, P.O.Box 68, 00014 Helsinki, Finland}
\email{jari.taskinen@helsinki.fi}

\thanks{The first two authors are supported in part by Ministerio de Ciencia e Innovaci\'on, Spain, project PID2022-136619NB-I00 and La Junta de Andaluc{\'i}a,
project FQM210. The third author was supported in part by the HORIZON-MSCA-2022-PF-01 project
101109510 LARGE BERGMAN  as well as by the Magnus Ehrnrooth Foundation and the 
V\"ais\"al\"a Foundation of the Finnish Academy of Sciences and Letters.}

\begin{abstract}
Let $\omega$ be a radial weight on the unit disc of the complex plane $\mathbb{D}$ and denote  $\omega_x =\int_0^1 s^x \omega(s)\,ds$, $x\ge 0$, for the moments of $\omega$
and $\widehat{\omega}(r)=\int_r^1 \omega(s)\,ds$ for the  tail integrals. A radial weight $\om$ belongs to the class $\widehat{\mathcal{D}}$ if  satisfies the upper doubling condition 
$$\sup_{0<r<1}\frac{\widehat{\omega}(r)}{\widehat{\omega}\left(\frac{1+r}{2}\right)}<\infty.$$
 If  $\nu$ or $\omega$ belongs to $\widehat{\mathcal{D}}$, it is described the boundedness of the Bergman projection $P_\omega$ induced by $\omega$ on the growth space $L^\infty_{\widehat{\nu}} 
=\{ f: \nm{f}_{\infty,v}=\text{\rm ess\,sup}_{z\in\D} |f(z)|\widehat{\nu}(z)<\infty\}$ in terms of neat conditions on the moments and/or the tail integrals of $\om$ and $\nu$. Moreover, it is solved the analogous problem for $P_\omega$ from 
$L^\infty_{\widehat{\nu}}$ to the Bloch type space 
$B^\infty_{\widehat{\nu}} = \set{f\, \text{analytic in $\mathbb{D}$}: \nm{f}_{B^\infty_{\widehat{\nu}}} = \sup_{z\in \D}(1-\abs{z})\widehat{\nu}(z)\abs{f'(z)}<\infty}.$ We also study similar questions 
for  exponentially decreasing radial weights. 
\end{abstract}

\maketitle
\section{Introduction and main results} \label{sec1}

The question of when the Bergman projection $P_\omega$ induced by a radial weight $\omega$ on the unit 
disc $\D$ of the complex plane, is a bounded operator in given function spaces 
is fundamental 
in the theory of spaces of analytic functions on $\D$. This is not only due to the mathematical difficulties the question raises, but also to its numerous applications in operator theory.  Indeed, bounded analytic projections can be used to establish duality relations and to obtain useful equivalent norms in spaces of analytic functions
\cite{HLS,PelRatproj,PR19,Zhu}. It has been recently described 
 the radial weights $\omega$ such that $P_\omega$ is bounded from $L^\infty$ to the classical  Bloch space $\mathcal{B}$ \cite{PR19} .  In this paper we are also interested in the natural limit 
  case $p = \infty$, and our main results provide complete characterizations of the radial
weights $\omega$ (respectively, $v$) such that  $P_\omega$ is bounded on the growth space  induced by $v$, 
when $v$ (resp. $\omega$) satisfies an upper doubling condition.
In addition, we will consider bounded projections onto weighted Bloch spaces in the setting of doubling
weights, and we also obtain new results on the boundedness of $P_\om$, when $\omega$ and $v$ belong to certain classes
of exponentially decreasing weights.  In order to present the precise statements of our results some definitions are needed.

Let $\H(\D)$ denote the space of analytic functions in $\D$.
 For a nonnegative function $\om\in L^1([0,1))$, the extension to $\D$, defined by 
$\om(z)=\om(|z|)$ for all $z\in\D$, is called a radial weight.
For $0<p<\infty$ and such an $\omega$, the weighted Lebesgue space $L^p_\om$ consists of complex-valued measurable functions $f$ on $\D$ such that
	$$
	\|f\|_{L^p_\omega}^p=\int_\D|f(z)|^p\omega(z)\,dA(z)<\infty,
	$$
where $dA(z)=\frac{dx\,dy}{\pi}$ is the normalized Lebesgue area measure on $\D$. The corresponding weighted Bergman space is $A^p_\om=L^p_\omega\cap\H(\D)$. Throughout this paper we assume $\widehat{\om}(z)=\int_{|z|}^1\om(s)\,ds>0$ for all $z\in\D$, for otherwise $A^p_\om=\H(\D)$. We also consider the Lebesgue space $L^\infty$ of complex-valued measurable functions  $f$ on $\D$ such that 
	$
	\nm{f}_{\infty} = \text{\rm{ess\,sup}}_{z\in \D} \abs{f(z)} <\infty
	$
and
the Hardy space is $H^\infty = L^\infty\cap \H(\D)$.

For a radial weight $\om$, the orthogonal Bergman projection $P_\om$ from $L^2_\om$ to $A^2_\om$ is
		\begin{equation*}
    P_\om(f)(z)=\int_{\D}f(\z) \overline{B^\om_{z}(\z)}\,\om(\z)dA(\z),
    \end{equation*}
where $B^\om_{z}$ are the reproducing kernels of $A^2_\om$. As usual,~$A^p_\alpha$ stands for the classical weighted
Bergman space induced by the standard radial weight $\omega(z)=(\alpha+1)(1-|z|^2)^\alpha$, $B^\alpha_z$ are the kernels of $A^2_\alpha$, and $P_\alpha$ denotes the corresponding Bergman projection.

One of the main obstacles  throughout this work is the lack of explicit expressions for
the Bergman reproducing kernel $B^\om_z$. For a radial weight $\om$, the kernel has the representation $B^\om_z(\z)=\sum \overline{e_n(z)}e_n(\z)$ for each orthonormal basis $\{e_n\}$ of $A^2_\om$, and therefore we are basically forced to work with the formula $B^\om_z(\z)=\sum_{n=0}^\infty\frac{\left(\overline{z}\z\right)^n}{2\om_{2n+1}}$ induced by the normalized monomials. Here $\om_{2n+1}$ are the odd moments of $\om$, and in general from now on we write $\om_x=\int_0^1r^x\om(r)\,dr$ for all $x\ge0$. Therefore the influence of the weight to the kernel is transmitted by its moments through this infinite sum, and nothing more than that can be said in general. This is in stark contrast with the neat expression $(1-\overline{z}\z)^{-(2+\alpha)}$ of the standard Bergman kernel $B^\alpha_z$ which is easy to work with.  
It is well known that for $1< p <\infty$ and $\alpha >-1$, the Bergman projection $P_\alpha$ acts as a bounded operator from $L^p_\alpha$ to $A^p_\alpha$  \cite[Section 4]{Zhu}. However, $P_\alpha$ is never bounded from $L^\infty$ to $H^\infty$. In fact, this is a  general phenomenon (known e.g. in the isomorphic
theory of Banach spaces) rather than a particular case.

	\begin{lettertheorem}\label{th: proyeccion no acotada}
There does not exist any bounded projection from  $L^\infty$ to $H^\infty$. In particular, 
		 $P_\om$ is not bounded from $L^\infty$ to $H^\infty$ for any radial weight.
	\end{lettertheorem}

In view of the previous result it is natural to look for a substitute of $H^\infty$ in the target space when the Bergman projection $P_\om$ acts on $L^\infty$.  
As for this question, it is known that the standard Bergman projection $P_\alpha$ is bounded and onto from $L^\infty$   to the Bloch space~$\mathcal{B}$, which consists of $f\in\H(\mathbb D)$ such that $\|f\|_{\mathcal{B}}=\sup_{z\in\D}|f'(z)|(1-|z|^2)+|f(0)|<\infty$ \cite[Theorem 5.2]{Zhu}. Furthermore,  it has been recently proved that the Bergman projection $P_\om$ induced by a radial weight $\om$ acts as a bounded operator from the space $L^\infty$ to 
$\mathcal{B}$ if and only if $\om\in\DD$, and $P_\om: L^\infty\to\mathcal{B}$ is bounded and onto if and only if $\om\in\DDD
=\DD\cap\Dd$ \cite[Theorems 1-2-3]{PR19}. 
Let us recall the reader that a radial weight $\om$ belongs to $\DD$,  and it is called upper doubling, if there exists $C=C(\om)>0$ such that 
$$
\omg(r)\le C\omg\left(\frac{1+r}{2}\right),\quad r\to1^-,
$$
and $\om\in\Dd$ if there exists $K=K(\om)>1$ and $C=C(\om)>1$ such that 
$$
\omg(r)\ge C\omg\left(1-\frac{1-r}{K}\right),\quad r\to1^-.
$$ 
We  say that $\omega$ is a doubling weight if $\om\in\DDD$.   Standard weights belong to $\DD$ but exponentially decreasing weights do not belong to $\DD$. Moreover, 
the containment in $\DD$ or $\Dd$ does not require differentiability, continuity or strict positivity. In fact, weights in these classes may vanish on a relatively large part of each outer annulus $\{z:r\le|z|<1\}$ of $\D$. For basic properties of the aforementioned classes, concrete nontrivial examples and more, see \cite{PelSum14,PelRat,PR19} and the relevant references therein.

A natural next step within this theory and the main objective of the paper is to study   the properties of a radial weight $\om$ which are determinative so that $P_\om$ is bounded on the  growth space
$L^\infty_v =\{ f\, \text{measurable}: \nm{f}_{\infty,v}=\text{\rm ess\,sup}_{z\in\D} |f(z)|v(z)<\infty\}$, where $v$ is a radial, continuous and decreasing weight  with $\lim_{r\to 1^-} v(r)=0$.
We denote $H^\infty_v = L^\infty_v \bigcap \H(\D)$.
Sections \ref{sec2}--\ref{sec4} deal with 
cases of doubling weights, and   in the previous definitions of spaces, 
$v$ will be equal to  $\nug$ for a radial weight $\nu$.  In Section \ref{sec4.1} we will study the boundedness of $P_\om$ from $L^\infty_{\nug}$ to $H^\infty_{\nug}$ under some hypothesis on $\nu$. In particular, assuming  $\nu\in\DD$
we will prove the following characterization of the pairs of radial weights 
$(\om,\nu)$ such that this fact happens.

	\begin{theorem}\label{th:caracterizacion nu Dgorro}
		Let $\om$ be a radial weight and $\nu \in \DD$. Then, the following conditions are equivalent:
		\begin{itemize}
			\item[\rm(i)] $P_\om : L^\infty_{\nug} \to H^\infty_{\nug}$ is bounded;
			\item[\rm(ii)]
 $\nu\in \Dd$ and 
			\begin{equation}\label{eq: desig caracterizacion momentos}
				\sup_{x\ge 0} \frac{\nu_x}{\om_{2x}}\left ( \frac{\om}{\nug}\right )_x <\infty;
			\end{equation}
			\item[\rm(iii)]  $\om \in \DD$, $\nu\in \Dd$ and
			\begin{equation}\label{eq: desig caracterizacion gorros}
			\sup_{ 0\le r < 1}  \frac{\nug(r)}{\omg(r)}	\int_r^1 \frac{\om(s)}{\nug(s)}\,ds  <\infty;
			\end{equation}
			\item[\rm(iv)] $\frac{\om}{\nug}\in \DDD$ and $\nu \in \Dd$.
		\end{itemize}
\end{theorem}

It is worth mentioning  that $\int_r^1 \frac{\om(s)}{\nug(s)}\,ds\ge \frac{\omg(r)}{\nug(r)}$ holds for all  radial weights and  that  \eqref{eq: desig caracterizacion gorros} 
means  nothing else but that the function $\frac{1}{\nug}$ satisfies a
 Bekoll\'e-Bonami type condition $B_{1,\om}$, see  \cite[Theorem~3]{PRW} for related results.
As for the proof of Theorem~\ref{th:caracterizacion nu Dgorro}, the condition \eqref{eq: desig caracterizacion momentos} is deduced from (i) using the 
the  reformulation 
\begin{equation}\label{eq:intro1}
	\sup_{a\in\D} \nug(a)\int_{\D} |B^\om_{a}(\z)|\,\frac{\om(\z)}{\nug(\z)}dA(\z)<\infty,
\end{equation}
together with a Hausdorff-Young inequality \cite[Theorem~6.1]{Duren}. Then it is proved that condition  \eqref{eq: desig caracterizacion momentos} implies that $\om\in\DD$ which joined to \cite[Theorem~1]{PelRatproj} allows to express the integral in \eqref{eq:intro1} in terms of  $\omg$ and $\nug$. Finally,  some technical calculations
imply that $\nu\in\Dd$ (see Proposition~\ref{prop: necesidad Hinfty nu en Dcheck} below) and we get (ii). The proofs of (ii)$\Leftrightarrow$(iii)$\Leftrightarrow$(iv) involve several descriptions of the classes of radial weight $\DD$ and $\Dd$. In order to prove (iii)$\Rightarrow$(i),    \cite[Theorem~1]{PelRatproj} is employed again.
	
In the case of standard weights $\omega (r) =(1-r^2)^\beta$, 
$\nu (r) = (1-r^2)^\alpha$, $\alpha, \beta > -1$, the projection $P_\omega$
is bounded on $L_{\nug}^\infty$, if and only if 
\begin{equation}
\beta > \alpha   . \label{eq: forellirudin}
\end{equation} 
This follows from the Forelli-Rudin estimates, see \cite{Zhu}, Lemma 3.10 and the proof 
of Theorem 3.12. Since both $\omega$ and $\nu$ belong to $\DDD$, we conclude that \eqref{eq: forellirudin} is an equivalent formulation of \eqref{eq: desig caracterizacion momentos} and \eqref{eq: desig caracterizacion gorros} in this simple case.

As a byproduct of Theorem~\ref{th:caracterizacion nu Dgorro} we deduce the next result which  might be expected in view of Theorem~\ref{th: proyeccion no acotada}, because roughly speaking the spaces 
$ L^\infty_{\nug}$ are "small" growth spaces for   $\nu \in \DD\setminus \Dd$.
\begin{corollary}\label{co: proyeccion no acotada pesos}
Let  $\nu \in \DD$. Then, there is a radial weight $\omega$ such that $P_\om: L^\infty_{\nug} \to H^\infty_{\nug}$ is bounded if and only if $\nu\in\Dd$.
In particular, if  $\nu \in \DD\setminus \Dd$ and a $\omega$ is a radial weight $P_\om$ is not bounded from $L^\infty_{\nug}$ to $H^\infty_{\nug}$.
\end{corollary}
Theorem~\ref{th:caracterizacion nu Dgorro} is also related  with the boundedness of the Szeg\"o projection. 
In Section \ref{sec3} we recall some literature references on the proof of Theorem 
\ref{th: proyeccion no acotada} and its relations to the boundedness
of the  Szeg\"o projection, since these can be used to give an alternative proof for one of our results
in Corollary \ref{co: proyeccion no acotada pesos}.

In Section \ref{sec4.2} we will study the boundedness of $P_\om$ from $L^\infty_{\nug}$ to $H^\infty_{\nug}$ assuming some hypothesis on $\om$. The following main results will
be proved using similar techniques to those employed in the proof of Theorem~\ref{th:caracterizacion nu Dgorro}.

\begin{theorem}\label{th:caracterizacion omg dgorro}
	Let  $\nu$ be a radial weight and $\om \in \DD$. Then, the following conditions are equivalent:
	\begin{itemize}
		\item[\rm(i)] $P_\om : L^\infty_{\nug} \to H^\infty_{\nug}$ is bounded;
		\item[\rm(ii)] $\nu \in \Dd$ and \eqref{eq: desig caracterizacion momentos} holds;
            \item[\rm(iii)] $\nu\in \Dd$ and \eqref{eq: desig caracterizacion gorros} holds;
		
		\item[\rm(iv)] $\frac{\om}{\nug}\in \DDD$ and $\nu \in \DDD$.
	\end{itemize}
\end{theorem}

\begin{corollary}\label{co: corolario caracterizacion omg dgorro}
Let $\om \in \DD$. Then there is a radial weight $\nu$ such that $P_\om : L^\infty_{\nug} \to H^\infty_{\nug}$ is bounded if and only if $\om\in\Dd$.
In particular, if  $\om \in \DD\setminus \Dd$ and a $\nu$ is a radial weight $P_\om$ in not bounded from $L^\infty_{\nug}$ to $H^\infty_{\nug}$.
\end{corollary}

Bearing in mind Theorem~\ref{th: proyeccion no acotada}, \cite[Theorem~1]{PR19} and our previous results  it is also interesting to study the boundedness from 
$L^\infty_\nu$ to the Bloch type space 
$$
B^\infty_{\nu} = \set{f\in \H(\D): \nm{f}_{B^\infty_{\nu}} = \sup_{z\in \D}(1-\abs{z})\nu(z)\abs{f'(z)}<\infty},
$$
where $\nu$ is a continuous and decreasing weight  with $\lim_{r\to 1^-}\nu(r)=0$. 
In Section \ref{sec4.3} we  will prove the following results for $\nu=\nug$.

\begin{theorem}\label{th:caracterizacion B nu Dgorro}
	Let $\om$ be a radial weight and $\nu \in \DD$. Then, the following conditions are equivalent:
	\begin{itemize}
		\item[\rm(i)] $P_\om : L^\infty_{\nug} \to \B^\infty_{\nug}$ is bounded;
		\item[\rm(ii)]   \eqref{eq: desig caracterizacion momentos} holds;
		\item[\rm(iii)] $\om \in \DD$ and \eqref{eq: desig caracterizacion gorros} holds.
	\end{itemize}
\end{theorem}

\begin{theorem}\label{th:caracterizacion B omg dgorro}
	Let $\nu$ be a radial weight and $\om \in \DD$. Then, the following conditions are equivalent:
	\begin{itemize}
		\item[\rm(i)] $P_\om:L^\infty_{\nug}\to \B^\infty_{\nug}$ is bounded;
		\item[\rm(ii)] \eqref{eq: desig caracterizacion momentos} holds;
		\item[\rm(iii)]\eqref{eq: desig caracterizacion gorros} holds.
	\end{itemize}
\end{theorem}

We point out  that on the contrary to the boundedness of  $P_\om:L^\infty_{\nug}\to H^\infty_{\nug}$, the boundedness of $P_\om:L^\infty_{\nug}\to \B^\infty_{\nug}$, $\om,\nu\in \DD$, does not require that $\om, \nu\in \Dd$.

Section \ref{sec5} is devoted to the study of the cases with exponentially decreasing 
weights, where we use techniques different from those for Theorems 
\ref{th:caracterizacion nu Dgorro}--\ref{th:caracterizacion B omg dgorro}. 
We postpone there the detailed definitions and the formulation of the results.

Finally, concerning the notation in this paper, we denote by $C$ (respectively $C(\omega)$ etc.) generic positive constants   independent of the variables involved (resp. depending only 
on $\omega$), the values of which may vary from place to place.
As usual,  $A\lesssim B$ ($B\gtrsim A$) for nonnegative functions $A$, $B$ means that $A\le C\,B$, for some  constant $C$. Furthermore, we write $A\asymp B$  when $A\lesssim B$ and $A\gtrsim B$.

\section{Preliminary results on radial weights} \label{sec2}
Throughout the paper we will employ different descriptions of the  weight classes  $\DD$, $\Dd$.
The next result gathers several characterizations of $\DD$ proved in \cite[Lemma~2.1]{PelSum14}. 

\begin{letterlemma}
	\label{caract. pesos doblantes}
	Let $\om$ be a radial weight. Then, the following statements are equivalent:
	\begin{itemize}
		\item[\rm(i)] $\om \in \DD$;
		\item[\rm(ii)] There exist $C=C(\om)\geq 1$ and $\a_0=\a_0(\om)>0$ such that
		$$ \omg(s)\leq C \left(\frac{1-s}{1-t}\right)^{\a}\omg(t), \quad 0\leq s\leq t<1,$$
		for all $\a\geq \a_0$;
		\item[\rm(iii)]
		$$ \om_x=\int_0^1 s^x \om (s) ds\asymp \omg\left(1-\frac{1}{x}\right),\quad x \in [1,\infty);$$
		\item[\rm(iv)] There exists $C(\om)>0$ such that $\om_x\le C \om_{2x}$, for any $x\ge 1$.
	\end{itemize}
\end{letterlemma}

In particular, we will repeatedly use Lemma~\ref{caract. pesos doblantes}(iii) and the fact that
$\sup_{x\ge 1}\frac{\widehat{\om}\left( 1-\frac{1}{x}\right)}{\om_x}<\infty$  for any radial weight $\omega$.  

We will also use  the following descriptions of the class $\Dd$, see \cite[Lemma 4]{MorPeldelaRosa23} and \cite[Lemma B]{PeldelaRosa22}.
\begin{letterlemma}
	\label{caract. D check}
	Let $\om$ be a radial weight. Then, the
	following statements are equivalent:
	\begin{itemize}
		\item[\rm(i)] $\om	\in \Dd$;
		\item[\rm(ii)] There exist $C=C(\om)>0$ and $\b=\b(\om)>0$ such that
		$$\omg(s)\leq C \left(\frac{1-s}{1-t}\right)^{\b}\omg(t), \quad 0\leq  t\leq s<1;$$
		\item[\rm(iii)] For each (or some) $\g >0$ there exists $C=C(\g, \om)>0$ such that
		$$ \int_{0}^r \frac{ds}{\omg(s)^{\g}(1-s)} \leq \frac{C}{ \omg(r)^{\g}},\quad 0\leq r < 1;$$
\item[\rm(iv)]  There exist $K(\om)>1$ and  $C=C(\om)>0$ such that
$$\int_{r}^{1-\frac{1-r}{K}} \om(s)\,ds\ge C \int_{r}^1 \om(s)\,ds,\quad 0\le r<1.$$
	\end{itemize}
\end{letterlemma}

We will also  use the following results.
\begin{lemma}
	\label{caract creciente}
Let $\om$ be a radial weight and let $\phi$ be a mapping $\phi: [0,1) \to 
(0, \infty)$. 
\begin{enumerate}
\item[\rm(i)] If $\om\in\DD$ and $\phi$ is a non-decreasing function, such that $\om\phi$ is a weight, then $\om\phi\in \DD$.
\item[\rm(ii)] If $\om\in\Dd$ and $\phi$ is a non-increasing function, then $\om\phi\in \Dd$.
\end{enumerate}
\end{lemma}
\begin{proof}
	(i). Observe that $\om \in \DD$ if and only if there exists $C(\om)>0$ such that
	$$
	\int_r^{\frac{1+r}{2}}\om(s)ds \leq C \int_{\frac{1+r}{2}}^1 \om(s)ds, \quad 0\leq r <1.
	$$
	Applying this, and the fact that $\phi$ is non-decreasing, we obtain that
	\begin{equation*}
		\begin{split}
			\frac{1}{\phi(\frac{1+r}{2})}\int_r^{\frac{1+r}{2}}\om(s)\phi(s) ds &\leq \int_r^{\frac{1+r}{2}}\om(s) ds 
\\ & \leq C \int_{\frac{1+r}{2}}^1 \om(s)ds \leq C \frac{1}{\phi(\frac{1+r}{2})}\int_{\frac{1+r}{2}}^1 \om(s)\phi(s) ds, \quad 0\leq r <1,
		\end{split}
	\end{equation*}
	then 
	$$
	\int_r^{\frac{1+r}{2}}\om(s)\phi(s)ds \leq C \int_{\frac{1+r}{2}}^1 \om(s)\phi(s)ds, \quad 0\leq r <1,
	$$
	so $\om\phi\in \DD$.

(ii). By Lemma~\ref{caract. D check}(iv), there holds $\om\in\Dd$ if and only if there are  $K>1$ and  $C(\om)>0$ such that
	$$
	\int_r^{1-\frac{1-r}{K}}\om(s)ds \ge C \int_{1-\frac{1-r}{K}}^1 \om(s)ds, \quad 0\leq r <1.
	$$
So, since $\phi$ is a non-increasing function,
\begin{equation*}
		\begin{split}
	&\int_r^{1-\frac{1-r}{K}}\om(s)\phi(s)\,ds \ge \phi\left(1-\frac{1-r}{K}\right) \int_r^{1-\frac{1-r}{K}}\om(s)\,ds
\\ & \ge
 C  \phi\left(1-\frac{1-r}{K}\right)\int_{1-\frac{1-r}{K}}^1 \om(s)ds
\ge  \int_{1-\frac{1-r}{K}}^1 \om(s)\phi(s)\,ds, \quad 0\leq r <1.
	\end{split}
	\end{equation*}
Hence $\om\phi\in \Dd$. This finishes the proof.
\end{proof}

\begin{lemma}\label{prop: producto en D}
	Let $\om \in \DDD$ and $\nu\in \DDD$.  Then $\om \nug\in \DDD$ and
\begin{equation}\label{eq:n1}
 \int_r^1\om(s)\nug(s)ds \asymp  \omg(r)\nug(r), \quad 0\le r<1.
\end{equation}
\end{lemma}
\begin{proof}
By Lemma~\ref{caract creciente}(ii), we have  $\om \nug\in \Dd$. 
On the other hand,  by Lemma~\ref{caract. D check}(iv)  there are  $K>1$ and  $C(\om)>0$ such that
	$$
	\int_r^{1-\frac{1-r}{K}}\om(s)ds \ge C \int_{r}^1 \om(s)ds, \quad 0\leq r <1.
	$$
So, Lemma~\ref{caract. pesos doblantes}(ii) together with the above inequality 
yield
$$ \int_{r}^1 \om(s)\nug(s)\,ds\ge \nug\left(1-\frac{1-r}{K} \right) \int_{1-\frac{1-r}{K}}^1\om(s)ds 
\gtrsim \nug(r)\omg(r), \quad 0\le r<1.$$
Therefore, using that $\om,\nu\in\DD$
$$
\int_r^1\om(s)\nug(s)ds \leq \omg(r)\nug(r) \lesssim \omg \left (\frac{1+r}{2} \right )\nug \left (\frac{1+r}{2} \right ) \lesssim \int_{\frac{1+r}{2}}^1\om(s)\nug(s)ds,
$$
so $\om\nug \in \DD$ and \eqref{eq:n1} holds. This finishes the proof.
\end{proof}

\section{Preliminary results on the boundedness of $P_\om: L^\infty_{\nug}\to H^\infty_{\nug}$}  \label{sec3}

\subsection{On the boundedness of the Szeg\"o projection in growth spaces} \label{sec3.1}

Some known results on  the boundedness of the Szeg\"o  projection  $R$ (also known as 
the Riesz projection) 
will be relevant for our studies on  Bergman projections. Recall that if $f$ is a 
complex valued harmonic function on 
$\D \ni z = r e^{i\theta}$ with a series representation $f(r e^{i \theta}) =  
\sum_{m=-\infty}^\infty f_m r^{|m|} e^{i m\theta}$ converging at least uniformly on every 
compact subset of $\D$, then the Szeg\"o projection $R$ by definition maps $f$ to the analytic 
function $\sum_{m=0}^\infty f_m r^m e^{i m\theta}$, where the series also converges at least 
uniformly on compact subsets. See \cite[Section 9.1]{Zhu}.

As for the known proofs of Theorem \ref{th: proyeccion no acotada}, if $P$ were a bounded 
projection from $L^\infty$ onto $H^\infty$, then one could consider the restriction of $P$ 
as a bounded projection from $L^\infty(\partial \D)$ onto  $H^\infty$. 
But it is well-known that such a projection does not exist. Indeed, by techniques 
explained in \cite[Theorem 5.18, Example 5.19]{Rud}, see also \cite[Theorem 9.7]{Zhu}, one 
can show that the existence of such a bounded projection would imply the untrue fact 
that the Szeg\"o projection $R: L^\infty(\partial \D) \to H^\infty$ is bounded (e.g.\cite{Zhu}, p. 258). 

Another way to see Theorem \ref{th: proyeccion no acotada} follows from the fact
that $H^\infty$ is known not to be isomorphic to the Banach space $\ell^\infty$  of all bounded sequences
(e.g. \cite[Theorem 1.1]{Lu2}). Recall that a closed subspace of a Banach space is called 
{\it complemented},  if there exists a bounded projection onto it. Now, by \cite{Pel}, \cite[p.111]{LF},
the space  $L^\infty$ 
is isomorphic to the Banach space $\ell^\infty$, which is a 
{\it prime} Banach space (\cite[Theorem 2.a.7]{LF}), i.e. all of its complemented subspaces 
are again isomorphic to  $\ell^\infty$. We thus find that $H^\infty$ is not complemented in
$L^\infty$. 

The next result is probably also known to experts, but we indicate the simple proof for the convenience of the reader.

\begin{proposition}\label{pr:Szego}
Let $\omega$ and $\nu$  be radial weights such that $P_\omega: L_{\widehat{\nu}}^\infty \to H_{\widehat{\nu}}^\infty$ is bounded. Then,
the restriction of $P_\omega$ onto the closed subspace $h_{\widehat{\nu}}^\infty$ of 
$L_{\widehat{\nu}}^\infty$, consisting of harmonic functions,  coincides with the
Szeg\"o projection and consequently, $R: h_{\widehat{\nu}}^\infty \to H_{\widehat{\nu}}^\infty$ is bounded. 
\end{proposition}

\medskip

\begin{proof}
Assume that $P_\omega: L_{\widehat{\nu}}^\infty \to H_{\widehat{\nu}}^\infty$ is bounded and  $f(r e^{i \theta}) =  \sum_{m=-\infty}^\infty f_m r^{|m|} e^{i m\theta} \in h_{\widehat{\nu}}^\infty \subset L_{\widehat{\nu}}^\infty$ and $z = r e^{i \theta} \in \D$. Bearing in mind that  $B^\om_z(\z)=\sum_{n=0}^\infty\frac{\left(\overline{z}\z\right)^n}{2\om_{2n+1}}$ we obtain
\begin{equation}
\begin{split}
P_\omega f(z) & = \int_\D \sum_{n=0}^\infty \frac{\left(z \overline{\z}\right)^n}{
2\om_{2n+1}} 
f(\zeta) \omega(\zeta) dA(\zeta)  
\\
& = \frac1\pi \int_0^1 \int_0^{2 \pi} \sum_{n=0}^\infty  \big( 2\omega_{2n+1} \big)^{-1} r^n e^{i n \theta} 
\sum_{m=-\infty}^\infty  
f_m s^{n+|m|+1}e^{i\varphi(m-n)}  \omega(s) d\varphi ds
\\
& =   \sum_{n=0}^\infty \big( \omega_{2n+1} \big)^{-1} f_n r^n e^{i n \theta} 
 \int_0^1 s^{2 n +1}   \omega(s)  ds =  \sum_{n=0}^\infty f_n r^n e^{i n \theta}  = Rf(z) . 
\end{split}
\end{equation}
This finishes the proof.
\end{proof}

Now, let us observe that the "only if part" of  Corollary~\ref{co: proyeccion no acotada pesos} can be  deduced from Proposition~\ref{pr:Szego} and  some results in \cite{Lu2}
(see also \cite{Lu1}). In fact, some calculations show that
$\nu\in\DD$ if and only if 
$$
 \sup_{n \in \mathbb{N}} \frac{\nug(1- 2^{-n})}{\nug(1- 2^{-n-1})} < \infty$$
and $\nu\in\Dd$ if and only if 
$$
 \inf_{k \in \mathbb{N}} \lim\sup_{n \in \mathbb{N}} \frac{\nug(1- 2^{-n-k})}{\nug(1- 2^{-n})} < \infty. 
$$
Moreover, Proposition~6.4 and the first lines of p. 26 of  \cite{Lu2} show that if 
$\nu\in\DD$, then the Szeg\"o projection $R$ is bounded from $h_{\nug}^\infty$ to $H_{\nug}^\infty$ if and only 
if $\nu\in\Dd$, which together with  Proposition~\ref{pr:Szego} proves  that if $\om$ is a radial weight and $\nu\in\DD\setminus\Dd$ then $P_\om$ is not bounded from $L^\infty_{\nug}$ to  
$H^\infty_{\nug}$.
Later on we will prove Corollary~\ref{co: proyeccion no acotada pesos} using  Theorem~\ref{th:caracterizacion nu Dgorro}.

\subsection{Preliminary results on the boundedness of 
$P_\om: L^\infty_{\nug}\to H^\infty_{\nug}$} \label{sec3.2}

In this section we will prove some preliminary results of their own interest, which will
 be used in the  proof of  Theorem~\ref{th:caracterizacion nu Dgorro}.
 
\begin{proposition}\label{prop: norma Hinf}
	Let $\om$ and $\nu$ be radial weights. Then, $P_\om: L^\infty_{\nug}\to H^\infty_{\nug}$ is bounded if and only if
	\begin{equation}\label{eq: norma Hinf}
	\sup_{a\in\D} \nug(a)\int_{\D} |B^\om_{a}(\z)|\,\frac{\om(\z)}{\nug(\z)}dA(\z)<\infty.
	\end{equation}
\end{proposition}
\begin{proof}
It is clear that $P_\om: L^\infty_{\nug}\to H^\infty_{\nug}$ is bounded if \eqref{eq: norma Hinf} holds. Reciprocally, assume that 
 $P_\om: L^\infty_{\nug}\to H^\infty_{\nug}$ is bounded  and 
		for each $a\in\D$, consider the function
		$$ f_a(\z)
		=\begin{cases} \frac{B^\om_{a}(\z)}{\nug(\z)\abs{B^\om_{a}(\z)}} & \text{ if $B^\om_{a}(\z)\neq 0$ }\\  0 & \text{if $B^\om_{a}(\z)= 0$}\end{cases}. 
		$$
		Then, for each $a\in\D$, $\|f_a\|_{L^\infty_{\nug}}=1$ and 
		\begin{equation*}\begin{split}
				\nug(|a|)\int_{\D} |B^\om_{a}(\z)|\, \frac{\om(\z)}{\nug(\z)}dA(\z) & =\nug(|a|) \int_{\D}f_a(\z) \overline{B^\om_{a}(\z)}\,\om(\z)dA(\z)
				\\ & \le \sup_{z\in\D} \nug(|z|) \left| \int_{\D}f_a(\z) \overline{B^\om_{z}(\z)}\,\om(\z)dA(\z) \right|\le \| P_\om\|\|f_a\|_{L^\infty_{\nug}}=\| P_\om\|. 
		\end{split}\end{equation*}
Therefore,   \eqref{eq: norma Hinf} holds. This finishes the proof.
\end{proof}

It is worth noticing that an argument  similar to that of Proposition~\ref{prop: norma Hinf} allows to give an alternative proof of the fact that
 $P_\omega$ is not  bounded from $L^\infty$ to $H^\infty$  for any radial weight $\omega$. In fact, arguing as in the the proof of  Proposition~\ref{prop: norma Hinf} it follows that
$P_\omega$ is  bounded from $L^\infty$ to $H^\infty$ if and only if
$$\sup_{a\in\D} \int_{\D} |B^\om_{a}(\z)|\, \om(\z) dA(\z)<\infty.$$
Moreover, using Hardy's inequality  \cite[Theorem 5.1]{Duren}, for each $a\in\D\setminus\{0\}$
		\begin{equation*}\begin{split}
				\int_{\D} |B^\om_{a}(\z)|\,\om(\z)dA(\z) & 
				= 2 \int_0^1 M_1(B^\om_{a},s) s\omega(s)\,ds
				\\ & \ge 2\pi \int_0^1 \left( \sum_{n=0}^\infty \frac{s^n|a|^n}{2(n+1)\omega_{2n+1}} \right) s\omega(s)\,ds
				\\ & = \pi \sum_{n=0}^\infty \frac{|a|^n}{(n+1)\omega_{2n+1}} \int_0^1 s^{n+1}\omega(s)\,ds
				\\ & = \pi \sum_{n=0}^\infty \frac{|a|^n\omega_{n+1}}{(n+1)\omega_{2n+1}}\ge 
				\pi \sum_{n=0}^\infty \frac{|a|^n}{n+1}=\frac{\pi}{|a|}\log\left( \frac{1}{1-|a|}\right), 
		\end{split}\end{equation*}
		Therefore, $\sup_{a\in\D} \int_{\D} |B^\om_{a}(\z)|\,\om(\z)dA(\z)=\infty$ and consequently   $P_\omega$ is not  bounded from $L^\infty$ to $H^\infty$ .

We will use Proposition~\ref{prop: norma Hinf} to  obtain a necessary condition for the boundedness of $P_\om:L^\infty_{\nug}\to H^\infty_{\nug}$ in terms of moments and tails of $\om$ and $\nu$.
\begin{proposition}\label{prop: condicion necesaria}
	Let $\om$ and $\nu$ radial weights such that $P_\om:L^\infty_{\nug}\to H^\infty_{\nug}$ is bounded.  Then, there exists $C(\om,\nu)>0$ such that
	\begin{equation}\label{eq: condicion momentos}
	\left (\frac{\om}{\nug}\right )_x \leq C\frac{\om_{2x}}{\nug\left ( 1-\frac{1}{x}\right )},\quad x\geq 1.
	\end{equation}
\end{proposition}
\begin{proof}
	By Proposition~\ref{prop: norma Hinf} there exists $C(\om,\nu)>0$, such that 
	$$
	\sup_{z\in \D} \nug(z) \int_\D \abs{B^\om_z(\z)}\frac{\om(\z)}{\nug(\z)}dA(\z) \leq C,
	$$
	and by applying \cite[Theorem 6.1]{Duren},
	\begin{equation}
		\begin{split}
			C &\geq  	\sup_{z\in \D} \nug(z) \int_\D \abs{B^\om_z(\z)}\frac{\om(\z)}{\nug(\z)}dA(\z) \asymp 	\sup_{z\in \D} \nug(z) \int_0^1 \frac{\om(r)}{\nug(r)}M_1(B^\om_z,r)dr \\
			&\gtrsim \sup_{z\in \D, n\in \N} \nug(z) \int_0^1 \frac{\om(r)}{\nug(r)}\frac{\abs{z}^nr^n}{\om_{2n+1}}dr \geq \nug\left (1-\frac{1}{n} \right )\frac{\left (\frac{\om}{\nug} \right )_n}{\om_{2n+1}}\left (1-\frac{1}{n} \right )^n \\
			&\gtrsim \nug\left (1-\frac{1}{n} \right )\frac{\left (\frac{\om}{\nug} \right )_n}{\om_{2n}}, \quad n\in \N, \quad n\ge 2.
		\end{split}
	\end{equation}
For $x\geq 1$, take $N\in \N$ such that $N\leq x < N+1$, then the previous inequality yields 
	$$
	\left (\frac{\om}{\nug} \right)_x \leq 	\left (\frac{\om}{\nug} \right )_N \leq C \frac{\om_{2N}}{\nug\left ( 1-\frac{1}{N}\right )} \leq C\frac{\om_{2(N+1)}}{\nug\left ( 1-\frac{1}{N}\right )} \leq C\frac{\om_{2x}}{\nug\left ( 1-\frac{1}{x}\right )}.
	$$
	This finishes the proof.
\end{proof}

\section{Boundedness of  Bergman projection on growth space induced 
by radial doubling weight.}  \label{sec4}


\subsection{Boundedness of $P_\om:L^\infty_{\nug}\to H^\infty_{\nug}$ assuming conditions on $\nu$.} 
\label{sec4.1}
Theorem~\ref{th:caracterizacion nu Dgorro} and Corollary~\ref{co: proyeccion no acotada pesos}  will be proved in this section. Prior to presenting their proofs, two preliminary results are needed.

\begin{proposition}\label{prop: necesidad dgorro}
	Let $\om$ be a radial weight and $\nu \in \DD$ such that \eqref{eq: desig caracterizacion momentos} holds. Then,  
 $\om \in \DD$ and $\om\nug\in \DD$.
\end{proposition}
\begin{proof}
	Let $x\geq 1$.  By applying Lemma~\ref{caract. pesos doblantes} (iv) and \eqref{eq: desig caracterizacion momentos}
	\begin{equation}
		\begin{split}
			\om_x &\lesssim \frac{1}{\nu_x}\int_0^1\om(s) \nug(s)s^{2x} ds = \frac{1}{\nu_x}\int_0^1\om(s) s^{2x}\int_s^1 \nu(t)dt \, ds \\
			&\leq \frac{1}{\nu_x}\int_0^1\om(s) s^{\frac{3x}{2}}\int_s^1 t^{\frac{x}{2}}\nu(t)dt \, ds \leq \om_{\frac{3x}{2}} \frac{\nu_{\frac{x}{2}}}{\nu_x} \lesssim \om_{\frac{3x}{2}},
\quad  x\ge 1.
		\end{split}
	\end{equation}
Hence, $ \om_x\lesssim   \om_{\frac{3x}{2}}$. So,  iteration of this inequality yields  $ \om_x\lesssim   \om_{\frac{3x}{2}}\lesssim   \om_{\frac{9x}{4}}\le \om_{2x}$. Therefore, 
	by Lemma~\ref{caract. pesos doblantes} (iv), $\om\in \DD$. \\
	Now, a similar argument, the fact  that $\om,\nu\in \DD$ and \eqref{eq: desig caracterizacion momentos} imply
	$$
	(\om\nug)_x \leq \om_{\frac{x}{2}}\nu_{\frac{x}{2}} \lesssim \om_x\nu_x \lesssim (\om\nug)_{2x}, \quad x \geq 1,
	$$
	so $\om\nug\in \DD$ by Lemma~\ref{caract. pesos doblantes} (iv).
\end{proof}

\begin{proposition}\label{prop: necesidad Hinfty nu en Dcheck}
	Let $\nu$ be a radial weight and $\om\in\DD$ such that $P_\om:L^\infty_{\nug} \to H^\infty_{\nug}$ is bounded.  Then, $\nu \in \Dd$.
\end{proposition}

\begin{proof}
	By Proposition~\ref{prop: norma Hinf}, there exists a constant $C>0$ such that
	$$
	\int_\D \abs{B^\om_z(\z)}\frac{\om(\z)}{\nug(\z)}dA(\z) \leq \frac{C}{\nug(z)}, \quad z \in \D.
	$$
	Now, by \cite[Theorem 1]{PelRatproj} and Fubini's theorem, for $\frac{1}{2}\leq \abs{z} < 1$
	$$
	\int_\D \abs{B^\om_z(\z)}\frac{\om(\z)}{\nug(\z)}dA(\z) \asymp \int_0^1\frac{\om(s)}{\nug(s)}\int_0^{\abs{z}s}\frac{dt}{\omg(t)(1-t)}ds = \int_0^{\abs{z}}\frac{1}{\omg(t)(1-t)}\int_{\frac{t}{\abs{z}}}^1\frac{\om(s)}{\nug(s)}ds\;dt,
	$$
Applying Lemma~\ref{caract. pesos doblantes} (ii), we obtain that 
	\begin{equation*}
		\begin{split}
			\int_\D \abs{B^\om_z(\z)}\frac{\om(\z)}{\nug(\z)}dA(\z) &\gtrsim \int_0^{\abs{z}}\frac{\omg(\frac{t}{\abs{z}})}{\omg(t)} \cdot \frac{dt}{\nug(\frac{t}{\abs{z}})(1-t)} 
			\geq \int_0^{2\abs{z}-1}\frac{\omg(\frac{t}{\abs{z}})}{\omg(t)} \cdot \frac{dt}{\nug(t)(1-t)} \\
			&\gtrsim \int_0^{2\abs{z}-1} \left ( \frac{1-\frac{t}{\abs{z}}}{1-t} \right )^\b \cdot \frac{dt}{\nug(t)(1-t)} \\
			&\geq \int_0^{2\abs{z}-1} \left ( \frac{1-\frac{2t}{1+t}}{1-t} \right )^\b \cdot \frac{dt}{\nug(t)(1-t)} \\
			&\asymp \int_0^{2\abs{z}-1} \frac{dt}{\nug(t)(1-t)}, \quad \abs{z} \geq \frac{1}{2}.
		\end{split}
	\end{equation*}
	 Then, we take   $|z| = \frac{1+r}{2}$ and observe that there is $C=C(\omega,\nu)>0$ such that
	$$
	\int_0^{r} \frac{dt}{\nug(t)(1-t)} \le \frac{C}{\nug \left ( \frac{1+r}{2} \right )}, \quad 0\leq r < 1.
	$$
	Therefore if $0\leq r \leq s <1$, then
	$$
	\frac{1}{\nug(r)} \log\frac{1-r}{1-s} = \frac{1}{\nug(r)}\int_r^s\frac{dt}{1-t} \leq \int_r^s\frac{dt}{\nug(t)(1-t)} \leq \frac{C}{\nug\left ( \frac{1+s}{2}\right )}.
	$$
Finally, take $K>2$  such that $\frac{\log\frac{K}{2}}{C}>1$ and choose $s=1-\frac{2}{K}(1-r)>r$. Then,
$\frac{1+s}{2} = 1-\frac{1-r}{K}$, and
	$$
	\nug(r)\geq \frac{\log\frac{K}{2}}{C}\nug\left (1-\frac{1-r}{K} \right ), \quad 0\le r<1.
	$$
	 This shows that   $\nu \in \Dd$ and finishes the proof.
\end{proof}

\begin{Prf}{\em{Theorem~\ref{th:caracterizacion nu Dgorro}.}}
We will prove (i)$\Rightarrow$(ii)$\Leftrightarrow$(iii)$\Leftrightarrow$(iv) and finally (iii)$\Rightarrow$(i).

Assume that  $P_\om:L^\infty_{\nug} \to H^\infty_{\nug}$ is bounded. Then,  \eqref{eq: desig caracterizacion momentos} holds by
	 Proposition~\ref{prop: condicion necesaria} and  Lemma~\ref{caract. pesos doblantes} (iii).  Next,  Proposition~\ref{prop: necesidad dgorro} implies that 
 $\om \in \DD$ and therefore
 $\nu \in \Dd$ by Proposition~\ref{prop: necesidad Hinfty nu en Dcheck}, so we get (ii).
Now, if (ii) holds, then Proposition~\ref{prop: necesidad dgorro} implies that 
 $\om \in \DD$. So, \eqref{eq: desig caracterizacion gorros} follows from \eqref{eq: desig caracterizacion momentos} and  Lemma~\ref{caract. pesos doblantes} (iii). That is,  (iii) is satisfied. 
Reciprocally, if (iii) holds, then 
  $\frac{\om}{\nug}\in \DD$  by Lemma~\ref{caract creciente},  and \eqref{eq: desig caracterizacion momentos} follows from 
 \eqref{eq: desig caracterizacion gorros} and Lemma~\ref{caract. pesos doblantes} (iii). Therefore we get (ii).
Now let us prove (iii)$\Rightarrow$(iv).
Firstly,   $\frac{\om}{\nug}\in \DD$  by Lemma~\ref{caract creciente}.   In order to prove that $\frac{\om}{\nug}\in \Dd$ and to make the notation easier, we denote $\mu = \frac{\om}{\nug}$. Then, \eqref{eq: desig caracterizacion gorros} is equivalent to the existence of $C>1$ such that  
	\begin{equation}\label{eq: caract con muuu}
	\mug(r) \leq C \frac{\omg(r)}{\nug(r)}, \quad 0\leq r <1.
	\end{equation}
	On the other hand,  an integration by parts yields
	$$
	\omg(r) = \nug(r)\mug(r) - \int_r^1 \nu(s)\mug(s) ds, \quad 0\leq r <1,
	$$
	so \eqref{eq: caract con muuu} is equivalent to the existence of $C\in (0,1)$ such that
	\begin{equation}\label{desig2}
		\frac{\int_r^1\nu(s)\mug(s)ds}{\nug(r)}\leq C\mug(r), \quad 0\leq r <1.
	\end{equation}
	By Lemma~\ref{caract. D check} (ii) there exists $C_2>0$ and $\a>0$ such that for every $K>1$
	$$
	\frac{\nug\left (1 - \frac{1-r}{K} \right )}{\nug(r)} \leq \frac{C_2}{K^\a},\quad 0\leq r < 1.
	$$
	Therefore, taking $K$ such that $\frac{C_2}{K^\a} < 1-C$,  we obtain that
	\begin{equation*}
		\begin{split}
			C\mug(r) &\geq  \frac{\int_r^1\nu(s)\mug(s)ds}{\nug(r)} \geq  \frac{\int_r^{1-\frac{1-r}{K}}\nu(s)\mug(s)ds}{\nug(r)} \geq \mug\left (1-\frac{1-r}{K} \right )\frac{\nug(r)-\nug\left ( 1-\frac{1-r}{K}\right )}{\nug(r)} \\
			&= \mug\left (1-\frac{1-r}{K} \right )\left ( 1-\frac{\nug\left ( 1-\frac{1-r}{K}\right )}{\nug(r)} \right )\geq \mug\left (1-\frac{1-r}{K} \right )\left (1-\frac{C_2}{K^\a} \right ),\quad 0\leq r <1.
		\end{split}
	\end{equation*}
Consequently,  $\mu\in \Dd$ and (iv) holds.
Conversely, if (iv) holds, we get $\om\in\DD$ by Lemma~\ref{prop: producto en D}. Moreover, 
since $\frac{\om}{\nug}\in \Dd$, Lemma~\ref{caract. D check}(iv) implies that 
there are $K>1$ and $C>0$ such that
	$$
	\int_r^1\frac{\om(s)}{\nug(s)}ds \leq C \int_r^{1-\frac{1-r}{K}}\frac{\om(s)}{\nug(s)}ds, \quad 0\leq r < 1. 
	$$
This together with Lemma~\ref{caract. pesos doblantes}(ii) give
	\begin{equation*}
		\begin{split}
			\omg(r)&=\int_r^1 \om(s) ds = \int_r^1 \frac{\om(s)}{\nug(s)}\nug(s)ds \geq \int_r^{1-\frac{1-r}{K}}\frac{\om(s)}{\nug(s)}\nug(s)ds \\
			&\geq \nug\left ( 1-\frac{1-r}{K}\right )\int_r^{1-\frac{1-r}{K}}\frac{\om(s)}{\nug(s)}ds \gtrsim \nug(r)\int_r^1 \frac{\om(s)}{\nug(s)}ds, \quad 0\leq r<1,
		\end{split}
	\end{equation*}
which yields (iii). \\

Finally assume that (iii) holds and let us prove that $P_\om:L^\infty_{\nug}\to H^\infty_{\nug}$ is bounded. By Proposition~\ref{prop: norma Hinf}, it is sufficient to prove that
	$$
	\sup_{a\in\D, \abs{a}\geq \frac{1}{2}. }\nug(a)\int_\D \abs{B^\om_a(z)}\frac{\om(z)}{\nug(z)} dA(z) <\infty.
	$$
	By \cite[Theorem 1]{PelRatproj} and Fubini's Theorem, this is equivalent
	with proving that there is $C>0$ such that
	$$
	\nug(a)\int_0^{\abs{a}}\frac{1}{\omg(t)(1-t)} \int_{\frac{t}{\abs{a}}}^1\frac{\om(r)}{\nug(r)}dr\, dt \leq C, \quad \abs{a}\geq \frac{1}{2}.
	$$
 By applying \eqref{eq: desig caracterizacion gorros} and Lemma~\ref{caract. D check} (iii)
	\begin{equation*}
		\begin{split}
		\nug(a)\int_0^{\abs{a}}\frac{1}{\omg(t)(1-t)} \int_{\frac{t}{\abs{a}}}^1\frac{\om(r)}{\nug(r)}dr\, dt &
\le \nug(a)\int_0^{\abs{a}}\frac{1}{\omg(t)(1-t)} \int_{t}^1\frac{\om(r)}{\nug(r)}dr\, dt 
\\ & \lesssim
\nug(a)\int_0^{\abs{a}}\frac{dt}{\nug(t)(1-t)} dt \lesssim 1,
		\end{split}
	\end{equation*}
	so $P_\om:L^\infty_{\nug}\to B^\infty_{\nug}$ is bounded. This finishes the proof.
\end{Prf} \\

\begin{Prf}{\em{ Corollary~\ref{co: proyeccion no acotada pesos}.}}	
If  $P_\om:L^\infty_{\nug}\to H^\infty_{\nug}$ is bounded then $\nu\in\Dd$ by Theorem~\ref{th:caracterizacion nu Dgorro}. Reciprocally, 
if $\nu\in\DDD$
 by Lemma~\ref{caract. pesos doblantes}(ii), there exist $C,\a>0$  such that
	$$
	\frac{1}{\nug(s)}\leq C \left (\frac{1-r}{1-s} \right )^\a\frac{1}{\nug(r)}, \quad 0\leq r \leq s < 1.
	$$
Therefore, if $\om(z)=(1-|z|)^\g$ with
	$\g> \a-1$, then
	$$
	\int_r^1\frac{\om(s)}{\nug(s)}ds \lesssim  \frac{(1-r)^\a}{\nug(r)}\int_r^1 (1-s)^{\g-\a}ds \asymp \frac{(1-r)^{\g+1}}{\nug(r)} \asymp \frac{\omg(r)}{\nug(r)},
	$$
that is  \eqref{eq: desig caracterizacion gorros} holds.
	Then,  $P_\om$ is bounded from $L^\infty_{\nug}$ to $H^\infty_{\nug}$ by Theorem~\ref{th:caracterizacion nu Dgorro}. This finishes the proof.
\end{Prf}


\subsection{Boundedness of $P_\om:L^\infty_{\nug}\to H^\infty_{\nug}$ assuming conditions on $\omega$.}  \label{sec4.2}

This section is dedicated to the proof of Theorem~\ref{th:caracterizacion omg dgorro} and Corollary~\ref{co: corolario caracterizacion omg dgorro} among other results.
\medskip

\begin{Prf}{\em{Theorem~\ref{th:caracterizacion omg dgorro}.}} 
We will prove (i)$\Leftrightarrow$(iii)$\Leftrightarrow$(ii) and (iii)$\Leftrightarrow$(iv).

If (i) holds, then $\nu \in \Dd$ by Proposition~\ref{prop: necesidad Hinfty nu en Dcheck} and
	\eqref{eq: condicion momentos} holds  by Proposition~\ref{prop: condicion necesaria}. So, using that $\om\in\DD$ and Lemma~\ref{caract. pesos doblantes}, we get \eqref{eq: desig caracterizacion gorros}. Hence, (iii) holds. Reciprocally, if (iii) holds, using
\eqref{eq: desig caracterizacion gorros} and the fact that $\om\in\DD$
	$$
	\nug(r) \lesssim \frac{\omg(r)}{\int_r^1\frac{\om(s)\,ds}{\nug(s)}} \lesssim\frac{\omg\left (\frac{1+r}{2} \right )}{\int_{\frac{1+r}{2}}^1\frac{\om(s)\,ds}{\nug(s)}}  \leq \nug\left ( \frac{1+r}{2}\right ), \quad 0<r<1.
	$$
That is, $\nu\in\DD$. So, (i) follows from
 Theorem~\ref{th:caracterizacion nu Dgorro}. Next, assuming  again that (iii) holds
 we have  $\nu\in\DD$.  Moreover,     $\frac{\om}{\nug}\in \DD$ by Lemma~\ref{caract creciente}(i).
	 Therefore \eqref{eq: desig caracterizacion momentos}  follows from   \eqref{eq: desig caracterizacion gorros}  and Lemma~\ref{caract. pesos doblantes}. So (ii) is satisfied. 
Reciprocally, if (ii) holds,  using that $\om\in\DD$,  Lemma~\ref{caract. pesos doblantes} and  \eqref{eq: desig caracterizacion momentos}, we get \eqref{eq: desig caracterizacion gorros} 
and thus also (iii).  Finally, if (iii) holds, arguing as above yields $\nu\in\DD$, 
and so  $\frac{\om}{\nug}\in \DDD$ by  Theorem~\ref{th:caracterizacion nu Dgorro}. Hence, (iv) holds.
 On the other hand, (iv)$\Rightarrow$(iii) follows from Theorem~\ref{th:caracterizacion nu Dgorro}.
This finishes the proof.
\end{Prf} \\

\begin{Prf}{\em{Corollary~\ref{co: corolario caracterizacion omg dgorro}.}}
Let  $\om\in\DD$.
If  $P_\om:L^\infty_{\nug}\to H^\infty_{\nug}$ is bounded, then  Theorem~\ref{th:caracterizacion omg dgorro} yields $\frac{\om}{\nug}\in\DDD$ and $\nu\in\DDD$,
hence, by Lemma~\ref{prop: producto en D}, $\om\in\DDD$. 
On the other hand, if $\om\in\DDD$, then, by Lemma~\ref{caract. pesos doblantes} (ii), 
there exists $0<\b$ such that
	$$
	\left ( \frac{1-r}{1-s}\right )^\b\omg(s) \lesssim\omg(r).
\quad 0\leq r \leq s<1.
	$$
So, in particular $\omg(s)\lesssim (1-s)^\b,\, 0\le s<1$. 
Now, take $0<\gamma < \b$ and consider the standard weight $\nu(r)=(1-r)^{\gamma-1}$.
We have $\nug(r)\asymp (1-r)^\gamma$ and 

	\begin{equation*}
		\begin{split}
			\int_r^1 \frac{\om(s)}{(1-s)^\gamma} &\lesssim \frac{\omg(r)}{(1-r)^\gamma} +  \int_r^1 \frac{\gamma\,\omg(s)}{(1-s)^{\gamma+1}}ds
\\ & \lesssim  \frac{\omg(r)}{(1-r)^\gamma} +
\frac{\omg(r)}{(1-r)^\b} \int_0^1 (1-s)^{\b-\gamma-1}ds 
			\lesssim \frac{\omg(r)}{(1-r)^\gamma}, \quad 0\leq r < 1.
		\end{split}
	\end{equation*}
Consequently,  by Theorem~\ref{th:caracterizacion omg dgorro}, $P_\om$ is bounded from 
$L^\infty_{\nug}$ to $H^\infty_{\nug}$.
\end{Prf}

\medskip

We finish this section proving that both weights,  $\om$ and $\nu$, have to be doubling if  the operator $P_\om : L^\infty_{\nug}\to H^\infty_{\nug}$ 
is bounded and  one of them is assumed to be upper doublng.
\begin{corollary}\label{cor: todo en D}
Let $\om$ and $\nu$ be radial weights such that  $P_\om : L^\infty_{\nug}\to H^\infty_{\nug}$ 
is bounded and  $\om\in\DD$ or $\nu\in\DD$. Then, both weights  $\om,\nu \in \DDD$.
\end{corollary}

\begin{proof}
Assume that $P_\om:L^\infty_{\nug}\to H^\infty_{\nug}$ is bounded and   $\om\in\DD$, then $\om\in\DDD$ by Corollary~\ref{co: corolario caracterizacion omg dgorro} and 
$\nu\in\DDD$ by Theorem~\ref{th:caracterizacion omg dgorro}. 
On the other hand, if  $P_\om:L^\infty_{\nug}\to H^\infty_{\nug}$ is bounded and   $\nu\in\DD$, then $\nu\in\DDD$ and
$\frac{\om}{\nug}\in\DDD$  by Theorem~\ref{th:caracterizacion nu Dgorro},
so by Lemma~\ref{prop: producto en D} $\om\in\DDD$.  This finishes the proof.
\end{proof}

\subsection{Boundedness of $P_\om:L^\infty_{\nug}\to \B^\infty_{\nug}$ .}
\label{sec4.3}

In this section, we will prove Theorem~\ref{th:caracterizacion B nu Dgorro} and Theorem~\ref{th:caracterizacion B omg dgorro}  among other results.

Firstly, by mimicking the proof of Proposition~\ref{prop: norma Hinf} we get the analogous result for the boundedness of 
 $P_\om:L^\infty_{\nug}\to \B^\infty_{\nug}$ .
\begin{proposition}\label{prop: norma Binf}
	Let $\om$ and $\nu$ be radial weights.  Then,   $P_\om:L^\infty_{\nug} \to \B^\infty_{\nug}$ is bounded if and only if 
	$$\sup\limits_{z\in \D} (1-\abs{z})\nug(z)\int_\D \abs{(B^\om_\z)'(z)}\frac{\om(\z)}{\nug(\z)}dA(\z)<\infty.$$ 
	
\end{proposition}
As a byproduct of Proposition~\ref{prop: norma Binf} we get the following.
\begin{proposition}\label{condi nec 2}
	Let $\om$ and $\nu$ radial weights. If $P_\om:L^\infty_{\nug} \to \B^\infty_{\nug}$ is bounded, then  \eqref{eq: condicion momentos} holds.
\end{proposition}
\begin{proof}
The proof follows the ideas from the proof of Proposition~\ref{prop: condicion necesaria}. We include the details for the sake of completeness. 
If  $P_\om:L^\infty_{\nug}\to\B^\infty_{\nug}$ is bounded, 
then Proposition~\ref{prop: norma Binf} and \cite[Theorem 6.1]{Duren}  yield
	\begin{equation*}
		\begin{split}
			\infty
			&>\sup_{z\in\D}(1-\abs{z})\nug(z)\int_\D\abs{(B_\z^\om)'(z)}\frac{\om(\z)}{\nug(\z)}\,dA(\z)
			\\ &\gtrsim\sup_{z\in\D, n\in \N}(1-\abs{z})\nug(z) \frac{n\int_0^1|z|^{n-1}r^{n+1}\frac{\om(r)}{\nug(r)}\,dr }{\om_{2n+1}}
\\ & \gtrsim \sup_{n\in \N}\left( 1-\frac{1}{n}\right)^{n-1}\nug\left( 1-\frac{1}{n}\right)\frac{\left(\frac{\om}{\nug}\right)_{n+1}}{\om_{2n+1}}
\\ & \gtrsim \nug\left( 1-\frac{1}{n}\right)\frac{\left(\frac{\om}{\nug}\right)_{n+1}}{\om_{2n+1}}, \quad n\ge 2.
		\end{split}
	\end{equation*}
So arguing as in the proof of Proposition~\ref{prop: condicion necesaria} one gets \eqref{eq: condicion momentos}. This finishes the proof.
\end{proof}

\begin{Prf}{\em{Theorem~\ref{th:caracterizacion B nu Dgorro}.}}
We will prove (i)$\Rightarrow$(ii)$\Rightarrow$(iii)$\Rightarrow$(i).

Assume that $P_\om:L^\infty_{\nug} \to \B^\infty_{\nug}$ is bounded. Then, Proposition~\ref{condi nec 2} together with Lemma~\ref{caract. pesos doblantes} imply 
\eqref{eq: desig caracterizacion momentos}.  Hence (ii) is satisfied. Next assuming that (ii) holds, 
Proposition~\ref{prop: necesidad dgorro} implies that $\om\in\DD$ and consequently \eqref{eq: desig caracterizacion gorros} follows from  
 \eqref{eq: desig caracterizacion momentos} and Lemma~\ref{caract. pesos doblantes}, so  we get (iii). 
Finally, if (iii) holds, then 	using that  $z(B_\z^\om)'(z)=\overline{\z(B_z^\om)'(\z)}$, 
and  \cite[Theorem~1]{PelRatproj}, we obtain
\begin{equation*}
		\begin{split}
			\int_\D \abs{(B^\om_\z)'(z)}\frac{\om(\z)}{\nug(\z)}dA(\z) &\lesssim \int_0^{\abs{z}}\frac{1}{\omg(t)(1-t)^2}\int_t^1\frac{\om(s)}{\nug(s)}ds\; dt 
			\\ &\lesssim \int_0^{\abs{z}}\frac{dt}{\nug(t)(1-t)^2} \lesssim \frac{1}{\nug(z)(1-\abs{z})}, \, \abs{z}\geq \frac{1}{2}.
		\end{split}
\end{equation*}
Consequently, (i) holds by Proposition~\ref{prop: norma Binf} and thus the proof is complete. 
\end{Prf}

\medskip

\begin{Prf}{\em{Theorem~\ref{th:caracterizacion B omg dgorro}.}}
We will prove (i)$\Rightarrow$(iiii)$\Leftrightarrow$(ii) and (iii)$\Rightarrow$(i).

Assuming that $P_\om:L^\infty_{\nug} \to \B^\infty_{\nug}$ is bounded,  
Proposition~\ref{condi nec 2} together with Lemma~\ref{caract. pesos doblantes} imply
\eqref{eq: desig caracterizacion gorros}. Next, if (iii) holds
$\frac{\om}{\nug}\in \DD$ by Lemma~\ref{caract creciente}. Moreover, using \eqref{eq: desig caracterizacion gorros} we get
	$$
	\nug(r) \lesssim \frac{\omg(r)}{\int_r^1\frac{\om(s)\,ds}{\nug(s)}} \lesssim\frac{\omg\left (\frac{1+r}{2} \right )}{\int_{\frac{1+r}{2}}^1\frac{\om(s)\,ds}{\nug(s)}}  \leq \nug\left ( \frac{1+r}{2}\right ), \quad 0<r<1.
	$$
That is, $\nu\in\DD$.  Now, bearing in mind   Lemma~\ref{caract. pesos doblantes}  we conclude that \eqref{eq: desig caracterizacion momentos} is satisfied. Reciprocally, if (ii) holds, 
\eqref{eq: desig caracterizacion gorros} follows from Lemma~\ref{caract. pesos doblantes}. Finally, if (iii) holds and $\om\in\DD$, arguing as above we get $\nu\in\DD$, so (i) follows from 
Theorem~\ref{th:caracterizacion B nu Dgorro}. This finishes the proof.
\end{Prf}

\medskip

Next, we will obtain some consequences from Theorem~\ref{th:caracterizacion B nu Dgorro} and Theorem~\ref{th:caracterizacion B omg dgorro}.

\begin{corollary}\label{cor: B_mu si y solo si  D-gorro}
	Let $\om$ and $\nu$ be radial weights such that  $P_\om : L^\infty_{\nug}\to \B^\infty_{\nug}$ is bounded. Then  $\om\in\DD$ if and only if $\nu\in\DD$.
\end{corollary}
\begin{proof} 

Assume that $P_\om:L^\infty_{\nug} \to \B^\infty_{\nug}$ is bounded. If $\nu\in\DD$, then $\om\in\DD$ by Theorem~\ref{th:caracterizacion B nu Dgorro}. On the other hand, if 
$\om\in\DD$, the proof of (iii)$\Rightarrow$(ii) in Theorem~\ref{th:caracterizacion B omg dgorro} gives that $\nu\in\DD$. This finishes the proof.
\end{proof}

Let us observe that Theorem~\ref{th:caracterizacion nu Dgorro} together with Theorem~\ref{th:caracterizacion B nu Dgorro} imply that
for any radial weight $\om$ and $\nu\in \DDD$, the operators $P_{\om}: L^\infty_{\nug}\to \H^\infty_{\nug}$ and $P_{\om}: L^\infty_{\nug}\to \B^\infty_{\nug}$
are simultaneously bounded. This fact can be also deduced from the following result
 \cite[Theorems 3.2, 3.3]{luski}.
\begin{lettertheorem}
	Let $\nu$ be a radial weight. Then, the following statements holds:
	\begin{itemize} 
	\item[\rm(i)] $H^\infty_{\nug}$  is continously embedded in  $\B^\infty_{\nug}$ if and only if  $\nu\in \DD$;
	\item[\rm (ii)]  $H^\infty_{\nug} = \B^\infty_{\nug}$ with equivalence of norms if and only if $\nu \in \DDD$.
	\end{itemize}
\end{lettertheorem}

However, we are able to prove that the aforementioned result on the boundedness of  $P_{\om}$ does not remain true for   $\nu\in\DD\setminus\Dd$.
\begin{corollary}\label{co: H-B-nu}
Let $\nu\in\DD\setminus\Dd$. Then, for any radial weight $\omega$, $P_\om$ is not bounded from $L^\infty_{\nug}$ to $H^\infty_{\nug}$. However, there exists a radial weight
$\omega_\nu$ such that $P_{\om_\nu}: L^\infty_{\nug}\to \B^\infty_{\nug}$ is bounded.
\end{corollary}
\begin{proof}
Since  $\nu\in\DD\setminus\Dd$, Corollary~\ref{cor: todo en D} implies that  $P_\om$ is not bounded from $L^\infty_{\nug}$ to $H^\infty_{\nug}$ for any radial weight $\omega$.
Now, consider $\omega_\nu=\nu\nug$, then 
$\omg(r)\asymp \left( \nug(r) \right)^2, \, 0\le r<1,$
and $\om\in\DD$.  So,
$\int_r^1 \frac{\om_\nu(s)}{\nug(s)}\,ds= \nug(r)\asymp  \frac{\widehat{\om_\nu}(r)}{\nug(r)}, \, 0\le r<1.$
Therefore, $P_{\om_\nu}: L^\infty_{\nug}\to \B^\infty_{\nug}$ is bounded, by Theorem~\ref{th:caracterizacion B nu Dgorro}. This finishes the proof.
\end{proof}

It is also worth mentioning that there exist $\om\in\DDD$ such that   $P_\om$ is not bounded from $L^\infty_{\nug}$ to $H^\infty_{\nug}$ but 
$P_{\om}: L^\infty_{\widehat{\nu_\om}}\to \B^\infty_{\widehat{\nu_\om}}$ is bounded. Take $\om=1$ and $\nu(r)=(1-r)^{-1}\left( \log \frac{e}{1-r}\right)^{-2}\in\DD\setminus\Dd$,
and  apply  Corollary~\ref{cor: todo en D} and Theorem~\ref{th:caracterizacion B nu Dgorro}

We are also able to prove the following result.

\begin{corollary}\label{co: H-B-om}
Let $\om\in\DD\setminus\Dd$. Then for any radial weight $\nu$, $P_\om$ is not bounded from $L^\infty_{\nug}$ to $H^\infty_{\nug}$. However there exists a radial weight
$\nu_\om$ such that $P_{\om}: L^\infty_{\widehat{\nu_\om}}\to \B^\infty_{\widehat{\nu_\om}}$ is bounded.
\end{corollary}
\begin{proof}
Since  $\om\in\DD\setminus\Dd$, Corollary~\ref{cor: todo en D} yields that   $P_\om$ is not bounded from $L^\infty_{\nug}$ to $H^\infty_{\nug}$ for any radial weight $\nu$.
Now, for $\nu_\om(r)=\frac{d}{dr}\left( \left( \widehat{\om}(r) \right)^{\frac{1}{2}}\right)$ we have
$ \widehat{\nu_\om}(r)\asymp\left( \widehat{\om}(r)\right)^{\frac{1}{2}},  \, 0\le r<1,$
and $\nu\in\DD$.  Moreover, 
$\int_r^1 \frac{\om(s)}{\widehat{\nu_\om}(s)}\,ds\asymp \left( \widehat{\om}(r) \right)^{\frac{1}{2}}\asymp  \frac{\omg(r)}{\widehat{\nu_\om}(r)}, \, 0\le r<1.$
Therefore, $P_{\om}: L^\infty_{\widehat{\nu_\om}}\to \B^\infty_{\widehat{\nu_\om}}$ is bounded by Theorem~\ref{th:caracterizacion B nu Dgorro}. This finishes the proof.

\end{proof}

\section{Study of the boundedness of the Bergman projections  
for exponentially decreasing weights} \label{sec5}

\subsection{Class of exponentially decreasing weights} \label{sec5.1}
In this section we consider the boundedness of Bergman projections in the space $L_v^\infty$
for a class of exponentially decreasing weights $v$. In order to formulate the results we need 
some more definitions. As in \cite{HLS}, we say that a continuous function $\rho: \D \to
\mathbb{R}$ belongs to the class $\mathcal{L}_0$, if $\lim_{|z| \to 1} \rho(z) =0$,
\begin{equation}
\sup_{z,\zeta \in \D, z \not= \zeta} \frac{|\rho(z) -\rho(\zeta)|}{|z-\zeta|} < \infty
\label{7.1}
\end{equation}
and  for every $\varepsilon > 0$ there exists a compact $E \subset \D$ such that
$|\rho(z) - \rho(\zeta)| \leq \varepsilon |z - \zeta|$ whenever  $z, \zeta \in 
\D \setminus E$. Furthermore, according to \cite{HLS}, a twice continuously differentiable,
real valued subharmonic function $\varphi$ on $\D $ is said to belong to the class $\mathcal{W}_0$, if 
$ \Delta \varphi > 0$ and there exists $\rho \in \mathcal{L}_0$ such that
\begin{equation}
\frac{1}{\sqrt{\Delta \varphi}} \asymp \rho.  \label{7.1a}
\end{equation}
on $\D$.

We say that a weight $v $ on $\D$ belongs to the class $\mathcal{E}$, if it is of the form
$v = e^{-\varphi}$ for some {\it radial} function $\varphi \in \mathcal{W}_0$. 
As examples of weights in this class, we mention 
\begin{equation}
v(z) = \exp  \big(- \alpha / (1-|z|^\ell)^\beta \big)   \label{7.2}
\end{equation}
where $\alpha, \beta, \ell  > 0$ are constants. For these weights one can take $\rho(z) =
(1- |z|)^{1+ \beta/2}$.

For the weight $v$ as in  \eqref{7.2} with $\ell, \beta =1$ and $\omega = v^2$, the 
boundedness of $P_\omega$ on $L_v^\infty$  was proved in \cite{ConstPel}. In the setting of 
exponentially decreasing weights satisfying the general condition $(B)$  of \cite{Lu2}, \cite{LuTa},
examples of bounded projections from $L_v^\infty$ onto $H_v^\infty$ were constructed in \cite{LuTa}. 
(These projections are not necessarily of the Bergman type. All weights \eqref{7.2} satisfy condition $(B)$, 
but there are many others, see the mentioned references.) We also mention the following  result  of \cite{HLS}.

\begin{lettertheorem}  \label{thE}
If $v = e^{-\varphi}$ with $\varphi \in \mathcal{W}_0$ and  $\omega = v^2$, then the projection $P_\omega$ is bounded from $L_v^\infty$ onto $H_v^\infty$. 
\end{lettertheorem}

We will prove in Theorem \ref{th5.2} a generalization of this resultfor the class $\mathcal{E}$ but it is first
worthwhile to recall known, related  negative results. Indeed, in \cite{Dos} it was shown that given
a weight $\omega$ as in \eqref{7.2} with $\ell=2$, the projection $P_\omega$ is bounded in $L_\omega^p$, $1 \leq p \leq \infty$,
if and only if  $p=2$. The result was extended in \cite{Zey} . Furthermore, in  \cite[Theorems 1 and 2]{BLT} 
it was shown that
given $v$ as in \eqref{7.2} and another weight $\omega = \exp\big( -\tilde \alpha /(1 - |z|^\ell )^{\tilde \beta} 
\big)$, then $P_\omega$ is unbounded in $L_v^\infty$, if
$\tilde \alpha \not= 2 \alpha$ or $\tilde \beta \not= \beta $.

\subsection{Study of the boundedness $P_\omega : L_v^\infty \to H_v^\infty$ for 
perturbed exponentially decreasing weights}
\label{sec5.2}
In view of Theorem \ref{thE} and the above mentioned negative results,  it is of interest to know if the projection $P_\omega$ is bounded in $L_v^\infty$ for some 
other,  ``smaller'' perturbations $v$ of the weight $\omega^{1/2}$. 
We will prove the following generalization, which is our main result concerning 
exponentially decreasing weights. 

\begin{theorem}  \label{th5.2}
Assume the weight $w= e^{- \varphi}$ belongs to $\mathcal{E}$ with $\varphi, \rho$ satisfying \eqref{7.1a}.
Moreover, let $\nu \in \DD$ and define the weights
$\omega = w^2$ and  $v = w \nug^{t} \rho^\sigma$, where $t\in\{-1,0,1\}$ and  $\sigma $ is any real number. Then, the projection 
$P_\omega$ is bounded from $L_v^\infty$  onto $H_v^\infty$.
\end{theorem}

Here, the factor  $\nug^{t} \rho^\sigma$ can be considered as a modest multiplicative perturbation of the exponential weight $w$.
For example, for a weight $w$ as in \eqref{7.2}, the perturbation could be
any negative or positive power of the boundary distance. 

\begin{proof}
We present the proof for $t=1$,  similar proofs work for $t\in\{-1,0\}$.
We will need some results of \cite{HLS}. The Bergman reproducing kernel $B_z^\omega$ in 
\begin{equation}
P_\omega f(z) = \int_\D f(\zeta) \overline{B_z^\omega(\zeta)} \omega(\zeta ) dA (\zeta)
\end{equation}
has by \cite[Theorem 3.2]{HLS} the upper bound 
\begin{equation} 
|B_z^\omega(\zeta)| \leq C \frac{e^{\varphi(z) + \varphi(\zeta)}}{\rho(z) \rho(\zeta)} e^{-\alpha d_\rho(z,\zeta) }   \label{5.3}
\end{equation}
 for some $\alpha>0$, and  the last factor also has for any $M > 0$ the upper bound
\begin{equation} 
e^{-\alpha d_\rho(z,\zeta) } \leq C(M) \Big( \frac{\min( \rho(z), \rho(\zeta) )}{|z-\zeta|} \Big)^M  ,
\label{5.4}
\end{equation}
 see formula (23) of the reference. We will choose a large enough $M$ later. In \eqref{5.4}, $d_\rho$ is defined according to  \cite{HLS} as the  distance function 
\begin{equation} 
d_\rho(z,\zeta) = \inf_\gamma \int_0^1 |\gamma'(t)| \frac{dt }{\rho(\gamma(t))},
\end{equation}
where $z,\zeta\in \D$ and the infimum is taken over all piecewise $C^1$ curves 
$\gamma:[0,1] \to \D$ with $\gamma(0) = z$, $\gamma(1) = \zeta$.

Let $\nu \in \DD$ and $v = w \nug \rho^\sigma = e^{-\varphi} \nug \rho^\sigma$ be as in the assumptions
of the theorem, and let $f \in L_v^\infty$ be such that $\Vert f \Vert_{\infty,v} \leq 1$. We can estimate for every 
$z \in \D$, by using \eqref{5.3} and $|f(\zeta)| \leq \rho(\zeta)^{- \sigma} e^{\varphi(\zeta)}/ \nug(\zeta)$,
\begin{equation} \begin{split}
&  v(z) |P_\omega f(z)|  \leq w(z) \rho(z)^{ \sigma}  \nug(z) \int\limits_\D |B_z^\omega(\zeta) | |f(\zeta)| 
e^{-2 \varphi(\zeta) } dA(\zeta) 
\\ \leq & 
w(z) \nug(z) \rho(z)^\sigma e^{\varphi(z) }  \int\limits_\D 
\frac{e^{\varphi(\zeta)}}{\rho(z) \rho(\zeta)} e^{-\alpha d_\rho(z,\zeta) }  
\frac1{e^{ \varphi(\zeta)} \nug(\zeta) \rho(\zeta)^{\sigma} }  dA(\zeta) 
\\ \leq & 
\nug(z)\rho(z)^{\sigma -1}   \int\limits_\D 
\frac{  e^{-\alpha d_\rho(z,\zeta) }  }{\nug(\zeta) \rho(\zeta)^{\sigma +1}  } dA(\zeta)  .  \label{5.6}
\end{split} \end{equation}
We continue by dividing the integration domain $\D$ into two subsets,
\begin{equation} 
\mathbb{A}_z =  \Big\{ \zeta  \in \D \, : \, |\zeta|  \geq \frac12 (1 + |z|) \Big\}   \ \ \ 
\mbox{and} \ \ \ \mathbb{B}_z = \D \setminus \mathbb{A}_z   . \label{5.8}
\end{equation}
Then, there holds, for $\zeta \in \mathbb{B}_z$,
\begin{equation} 
\nug (\zeta) \geq \nug \Big(  \frac12 (1 + |z|) \Big) \geq C \nug (z) ,
\end{equation} 
since  $\nug$ is decreasing with respect to the radius and $\nu \in \DD$.
This yields  
\begin{equation} \begin{split}
&  \nug(z) \rho(z)^{\sigma -1 }  \int\limits_{\mathbb{B}_z} 
\frac{  e^{-\alpha d_\rho(z,\zeta) }  }{\nug(\zeta)\rho(\zeta)^{\sigma +1 } } dA(\zeta) 
\leq
C \rho(z)^{\sigma -1 }    \int\limits_{\mathbb{B}_z}
\frac{  e^{-\alpha d_\rho(z,\zeta) }  }{\rho(\zeta)^{\sigma+1} } dA(\zeta) 
\\ \leq & 
C \rho(z)^{\sigma -1 }    \int\limits_{\D}
\frac{  e^{-\alpha d_\rho(z,\zeta) }  }{\rho(\zeta)^{\sigma+1} } dA(\zeta) .  
\label{5.10}
\end{split} \end{equation}
By Corollary 3.1 of \cite{HLS}, this expression is bounded by a constant. 

To estimate the integral over $\mathbb{A}_z$ we first note that since $\nu \in \DD$,
Lemma \ref{caract. pesos doblantes}(ii) shows that there are constants $C,\beta > 0$ such that  
\begin{equation}
\frac{\nug(z)}{\nug(\zeta) } \leq C \frac{(1-|z|)^\beta}{(1-|\zeta|)^\beta }  
\nonumber 
\end{equation}
for all $z, \zeta$ with $|z| \leq |\zeta|$. 
We use this and  \eqref{5.4} to write
\begin{equation} 
\begin{split}
& \nug(z)  \rho(z)^{\sigma-1}   \int\limits_{\mathbb{A}_z }  \frac{  e^{-\alpha d_\rho(z,\zeta) }  }{\nug(\zeta)\rho(\zeta)^{\sigma+1} } dA(\zeta) 
\leq 
C (1 - |z|)^\beta \rho(z)^{\sigma-1}   
\int\limits_{\mathbb{A}_z }  \frac{  e^{-\alpha d_\rho(z,\zeta) }  }{
\rho(\zeta)^{\sigma+1} (1- |\zeta|)^\beta }  dA(\zeta) 
\\ \leq & 
C(M) (1 - |z|)^\beta  \rho(z)^{\sigma-1}   \int\limits_{\mathbb{A}_z}
\frac{ 1  }{(1 - |\zeta|)^\beta\rho(\zeta)^{\sigma+1} }  
\bigg( \frac{ \min\big(  \rho(z), \rho(\zeta) \big) }{|z-\zeta|} \bigg)^M  dA(\zeta) 
\label{5.12}
\end{split} 
\end{equation}
We now choose the number $\gamma > 0$ such that $\sigma + \gamma -1 > 0$ and then  $M$ such that
$M > \sigma + \beta +  \gamma + 3 $. Then, we use
$$
\min\big(  \rho(z), \rho(\zeta) \big)^M \leq \rho(z)^\gamma \rho(\zeta)^{M-\gamma}
$$
to bound \eqref{5.12} by
\begin{equation} 
C_M (1 - |z|)^\beta  \rho(z)^{\sigma + \gamma -1}   \int\limits_{\mathbb{A}_z}
\frac{ \rho(\zeta)^{M - (\gamma + \sigma +1)} }{(1 - |\zeta|)^\beta |z-\zeta|^M}   dA(\zeta)  .
\label{5.14}
\end{equation}
Note that by \eqref{7.1}, the function $\rho$ has the bound $\rho(\zeta) \leq C(1- |\zeta|)$ for all $\zeta \in \D$. Thus, \eqref{5.14} is not larger than
\begin{equation} 
C_M (1 - |z|)^{\sigma + \beta +\gamma -1}   \int\limits_{\mathbb{A}_z}
\frac{ (1 - |\zeta|)^{M - (\sigma + \beta + \gamma +1)} }{|z-\zeta|^M}   
dA(\zeta)  . \label{5.16}
\end{equation}
Since $|z-\zeta| \geq C | 1 - z \bar \zeta|$ for $\zeta \in \mathbb{A}$, the Forelli-Rudin estimates 
\cite[Lemma 3.10]{Zhu} and the choice of the number $M$ show that 
\eqref{5.16} is bounded by a constant.

Combining \eqref{5.6} and \eqref{5.10}--\eqref{5.16} yield
$\sup_{z \in \D} v(z) |P_\omega f(z)|  \leq C$ for all $f\in L_v^\infty$ with 
$\Vert f \Vert_{\infty,v} \leq 1$, which proves the theorem. \ \ $\Box$
\end{proof}

Finally, we will apply Theorem~\ref{thE} in order to get the following result for the 
special exponential weights and set $w= e^{-\varphi}$, 
$\varphi(z) = \alpha (1-|z|^2)^{- \beta}$,
and 
\begin{equation}
v(z) = (1 - |z|^2)^\gamma w(z) , \ \  \ \ 
\omega(z)  = (1- |z|^2 )^{2 \sigma} w(z)^2   , \label{5.34}
\end{equation}
where the parameters satisfy $\alpha, \beta > 0 $, $\sigma, \gamma \in \mathbb{R}$.

\begin{proposition}  \label{prop22}
Let the weights $v$ and $\omega$ be as in \eqref{5.34}. Then,
the projection $P_\omega$ is bounded from $L_v^\infty$ onto $H_v^\infty$.
\end{proposition}

\begin{proof}
We define
\begin{equation}
\psi(z) = \frac{\alpha}{(1-|z|^2)^\beta} - \sigma \log ( 1 - |z|^2)  \nonumber
\end{equation}
and observe that $\psi \in \mathcal{W}_0$, since there holds $1 / \sqrt{\Delta \psi(z)}
\asymp R(z)$ for the function $R(z) = (1-|z|^2)^{ 1 + \beta/2} \in \mathcal{L}_0$.
We define the weight $W= e^{-\psi} \in \mathcal{E}$. Then, obviously, $\omega= W^2$ and 
$
v(z) = R(z)^{\frac{\gamma-\sigma}{1 + \beta/2} } W(z) 
$ 
so that the result follows from Theorem \ref{th5.2} by choosing $W \to w $ and
$R \to \rho $ and setting $t=0$. 
\end{proof}


\begin{thebibliography}{99}

\bibitem{BLT} J.~Bonet, W.~Lusky and J.~Taskinen, Unbounded Bergman projections on
weighted spaces with respect to exponential weights, Integr. Equ. Oper. Theory (2021) 61--93.

\bibitem{ConstPel}	O.~Constantin and J.~A.~Pel\'aez, Boundedness of the Bergman projection on $L^p$-spaces with exponential weights,
Bull. Sci. Math. 139 (2015) 245--268.
	
	
\bibitem{Dos} M.~Dostanic, Unboundedness of the Bergman projections on $L^p$-spaces with 
exponential weights, Proc. Edinb. Math. Soc. 47 (2004) 111--117.

\bibitem{Duren} 
P.~Duren, Theory of $H^p$ spaces. Pure and Applied Mathematics, Vol. 38 Academic Press, New York-London 1970 xii+258 pp. 





\bibitem{HLS} Z.~Hu, X.~Lv and  A.~P.~Schuster, Bergman spaces with exponential weights,
J. Funct. Anal. 276 (2019), no.5, 1402--1429.

\bibitem{LF} J.~Lindenstrauss and L.~Tzafriri, Classical Banach spaces 1. Ergebnisse der Mathematik und ihrer Grenzgebiete 92, Springer (1977).

\bibitem{Lu1}  W.~Lusky, 
On weighted spaces of harmonic and holomorphic functions, J. London. Math. Soc (2) 51 (1995), 309--320.

\bibitem{Lu2}  W.~Lusky, 
On the isomorphism classes of weighted spaces of harmonic and holomorphic functions, Studia Math. 175 (2006), no. 1, 19--45.

\bibitem{LuTa} W.~Lusky and J.~Taskinen, Bounded holomorphic projections for
exponentially decreasing weights, J. Function Spaces Appl. 6 (2008), no. 1, 59--70.


\bibitem{MorPeldelaRosa23} 
A.~M.~Moreno, J.~A.~Pel\'aez and E. de la Rosa, Fractional derivative description of the Bloch space,
		Potential Anal. https://doi.org/10.1007/s11118-023-10119-z.

\bibitem{PelSum14} J.~A.~ Pel\'aez,
Small weighted Bergman spaces, Proceedings of the summer school in Complex and Harmonic analysis and related topics, (2016).

\bibitem{PelRat}
J.~A.~ Pel\'aez and J. R\"atty\"a, Weighted Bergman spaces induced by rapidly increasing weights, Mem. Amer. Math. Soc. 227 (2014), no.~1066.
						
		
\bibitem{PelRatproj} J.~A.~ Pel\'aez and J. R\"atty\"a,
Two weight inequality for Bergman projection, J. Math. Pures Appl. 105 (2016), 102--130.
		
\bibitem{PR19} J.~A.~ Pel\'aez and J. R\"atty\"a, Bergman projection induced by radial weight. Adv. Math. 391 (2021), Paper No. 107950, 70 pp.

\bibitem{PeldelaRosa22} J.~A.~Pel\'aez and E. de la Rosa, Littlewood-Paley inequalities for fractional derivative on Bergman 
spaces. Ann. Fenn. Math.  47 (2022), no. 2, 1109--1130.


\bibitem{PRW} J.~A.~ Pel\'aez, J. R\"atty\"a and B.~D.~Wick, Bergman projection induced by kernel with integral representation,
J. Anal. Math. 138 (2019), no.1, 325--360.

\bibitem{Pel} A. Pe{\l}czy\'nski, On the isomorphism of the spaces $M$ and $m$,  Bull. Acad. Pol. Sci. 6 (1958), 695--696.

\bibitem{Rud} W.~Rudin, Functional analysis, McGraw-Hill (1973).

\bibitem{Zey} Y.~E.~Zeytuncu, $L^p$-regularity of weighted Bergman projections, 
Trans. Am. Math. Soc. 365 (2013) 2959--2976.

\bibitem{Zhu}  K.~Zhu, Operator theory in function spaces, 2nd. Mathematical Surveys and
Monographs, vol. 138, 2nd edn. American Mathematical Society, Providence
(2007)
		
	\end{thebibliography}
	\end{document}